\documentclass[preprint,12pt]{elsarticle}
\usepackage[english]{babel}
\usepackage{amsmath,amsthm,amssymb,scrextend}
\usepackage{graphicx}
\usepackage[colorlinks=true, allcolors=blue]{hyperref}
\usepackage{xargs}                     
\usepackage{xcolor}  
\usepackage{enumerate}
\usepackage[colorinlistoftodos,prependcaption,textsize=tiny]{todonotes}
\newcommandx{\comment}[2][1=]{\todo[linecolor=red,backgroundcolor=red!25,bordercolor=red,#1]{#2}}

\newtheorem{lemma}{Lemma}

\newtheorem{proposition}{Proposition}
\newtheorem{example}{Example}
\journal{ }

\begin{document}

\begin{frontmatter}
\title{On the Initial Algebra and Final Co-algebra of some Endofunctors on  Categories of  Pointed Metric Spaces}

\author[and]{Annanthakrishna Manokaran} \ead[and]{annanth84@gmail.com}
\author[rom]{Romaine Jayewardene} \ead{romaine.jayewardene@gmail.com} 
\author[rom]{Jayampathy Ratnayake} \ead{jratnaya@maths.cmb.ac.lk} 
\address[and]{Department of Mathematics and Statistics, University of Jaffna, Jaffna, Sri Lanka}
\address[rom]{Department of Mathematics, University of Colombo, Colombo 3, Sri Lanka}

\begin{abstract}
We consider two endofunctors of the form  $~F:X\longrightarrow M\otimes X~$, where $~M~$ is a non degenerate module, related to the unit interval and the Sierpinski gasket, and their final co-algebras. The functors are defined on the categories of bi-pointed and tri-pointed metric spaces, with continuous maps, short maps or Lipschitz maps as the choice of morphisms.   

First we demonstrate that the final co-algebra for these endofunctors on the respective category of pointed metric spaces with the choice of continuous maps is the final co-algebra of that with short maps and after forgetting the metric structure is of that in the set setting. We use the fact that the final co-algebra can be obtained by a Cauchy completion process, to construct the mediating morphism from a co-algebra by means of the limit of a sequence obtained by iterating the co-algebra. We also show that the Sierpinski gasket $~(\mathbb{S},\sigma)~$ is not the final co-algebra for these endofunctors when the morphism are restricted to being Lipschitz  maps.  
\end{abstract}


\end{frontmatter}
\section{Introduction}\label{introduction}
This paper considers  the  initial algebra and final co-algebra for two particular endofunctors $~F_i : \mathbf{C}\rightarrow \mathbf{C}~$, where $~\mathbf{C}~$ is the category of $i$-pointed  metric spaces with continuous maps ($~\mathbf{MS_{i}}^{C}~$), short maps ($~\mathbf{MS_{i}}^{S}~$) or  Lipschitz maps ($~\mathbf{MS_{i}}^{L}~$), and  $i=2,3$. The two functors are based on the unit interval $[0, 1]$ and the Sierpinski gasket. These definitions are motivated by \cite{Lein} and have been considered previously in \cite{Bhat} and \cite{Bhat2}. We need some definitions to start.\\

A bi-pointed set is a set having two distinguished elements. We denote such a set by a triple $~(X,\top,\bot)~$, where $\top$ and $\bot$ are the two distinguished elements. In a similar manner, a tri-pointed set, denoted by a quadruple  $~(X,T,L,R)~$, consists of a set and three distinguished elements. We will often omit the distinguished points from the description of the set and simply write ``Let $~X~$ be a tri-pointed set...". To differentiate distinguished points of two (or more) $i$-pointed sets $X$ and $Y$, we will use subscripts, such as $\bot_X$ and $~\top_Y~$ for $~\bot~$ of $~X~$ and $~\top~$ of $~Y~$ respectively (in the case of $i=2$).\\  

One can similarly define an $i$-pointed set for any $i \in \mathbb{N}$. There is a category of $i$-pointed sets, $~\mathbf{Set_{i}}$, whose objects are $i$-pointed sets and morphisms are functions which preserve the distinguished elements. An $i$-pointed metric space $~(X, d)~$ is an $i$-pointed set $~X=(X, x_1,...,x_i)~$ equipped with a one-bounded metric ($~d(x,y) \leq 1 \,\,\, \forall\, x,y\in X$) such that the distance between any pair of distinguished elements is 1. The class of $i$-pointed metric spaces can be raised to the categories $~\mathbf{MS_{i}}^{S}~$, $~\mathbf{MS_{i}}^{L}~$ and $~\mathbf{MS_{i}}^{C}~$, where the morphisms are respectively short maps, Lipschitz maps and continuous maps that preserves the distinguished elements. Note that  $~\mathbf{MS_{i}}^{L}~$ and $~\mathbf{MS_{i}}^{S}~$ are subcategories of $~\mathbf{MS_{i}}^{C}~$,  $~\mathbf{MS_{i}}^{S}~$ is a subcategory of $~\mathbf{MS_{i}}^{L}~$ and all these three categories ($~\mathbf{MS_{i}}^{S}~,~\mathbf{MS_{i}}^{L}~,~\mathbf{MS_{i}}^{L}~$) are subcategories of $~\mathbf{Set_{i}}~$. In this paper we will consider only the cases $i=2$ and $i=3$.\\

The functors $~F_{i} ~$ were defined, for example, in \cite{Bhat}, and we invite the reader to refer it for details. The functors $F_i$ are defined at the level of $~\mathbf{Set_{i}}~$ by $~F_i X_i =M_{i}\otimes X_i~$, for a particular set $M_i$. Here $~M_{i}\otimes X_i~$ is the set of equivalence classes of a particular equivalence relation defined on $~M_{i}\times X_i~$. We will denote the equivalence  class of an element $~(m, x)\in M_i \times X_i$ by $~m\otimes x~$. Details for the two cases $i=2$ and $i=3$ are given below.\\

For the bi-pointed case, we take $~M_{2}=\{l,r\}~$ and consider the equivalence relation on  $~M_{2}\times X~$ generated by the relation $~(l,\top)~\sim~(r,\bot)$. The set $~M_{2}\otimes X~$ is lifted to a bi-pointed set by choosing $~l\otimes \bot~$ and  $~r\otimes \top~$ as $~\bot_{M_2\otimes X}~$ and $~\top_{M_2\otimes X}~$ respectively. This description is based on Freyd's description of the unit interval $~[0,1]~$ as a final co-algebra (see \cite{Bhat}).\\

For the tri-pointed case, we take $~M_{3}=\{a,b,c\}~$ and consider the equivalence relation on  $~M_{3}\times X~$ generated by the relations $~(b,T)~\sim~(a,L)$, $~(a,R)~\sim~(c,T)~$ and $~(c,L)~\sim~(b,R)~$. The set $~M_{3}\otimes X~$ is lifted to a tri-pointed set by choosing $~a\otimes T~$, $~b\otimes L~$ and  $~c\otimes R~$ as $~T_{M_3\otimes X}~$, $~L_{M_3\otimes X}~$ and $~R_{M_3\otimes X}~$ respectively. This description glues three copies of $~X~$ as shown in the diagram and is based on the Sierpinski gasket (see \cite{Bhat}).\\

\begin{figure}
\centering
\includegraphics[scale=0.5]{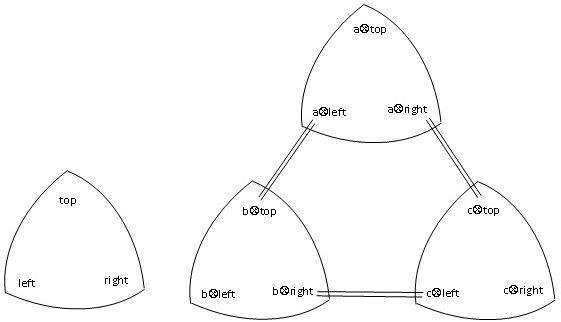}
\caption{Description of $F_3\,X$}
\end{figure}

Given a morphism $~f:X\rightarrow Y~$ of $i$-pointed sets, $~F_{i}f : M_{i}\otimes X\rightarrow  M_{i}\otimes Y~$ is given by $~F_if(m\otimes x)=m\otimes f(x)~$. This definition makes $~F_{i}~$ an endofunctor on $~\mathbf{Set_{i}}~$. One can easily verify that $~F_if~$ preserves the distinguished elements. For example, in the bi-pointed case we have $~F_{2}f(l\otimes \bot_{X})=l\otimes f(\bot_{X})=l\otimes \bot_{Y}~$ and $~F_{2}f(r\otimes \top_{X})=r\otimes f(\top_{X}) = r\otimes \top_{Y}~$.\\

Moreover, the above definitions give rise to endofunctors on $~\mathbf{MS_{i}}^{S}~$,   $~\mathbf{MS_{i}}^{L}~$ and $~\mathbf{MS_{i}}^{C}~$, where $i=2,3$; which we also identify by  $~F_{i}~$, as follows. First, for a given $i$-pointed metric space $~(X_{i},d)~$, $~M_{i}\times X_{i}~ $ is given the metric 

	$$d\left((m,x),(n,y)\right)=\left\{%
	\begin{array}{ll}
	\frac{1}{2}d(x,y), & \hbox{$m=n$;} \\\\
	1, & \hbox{$m\neq n$.}
	\end{array}%
	\right.$$ 

We now consider the quotient metric on $~M_{i}\otimes X_{i}~$. Though the quotient metric in general is only a pseudo-metric, in our case it is indeed a metric and the distance between two elements can be computed explicitly as follows (see \cite{Bhat} for proofs). 

	
  \begin{lemma} 
  \label{LemmaMetricOnTensor}
  The distance $~d(m\otimes x, n\otimes y)~$ is calculated by the following formulas.  
\\For the bi-pointed case: 
$$d(m\otimes x,n\otimes y)=\left\{%
	\begin{array}{ll}
	\frac{1}{2}d(x,y), & \hbox{$m=n$;} \\\\
	\frac{1}{2}d(x,\top)+\frac{1}{2}d(\bot,y), & \hbox{$m\neq n$.}
	\end{array}%
	\right.$$
For the tri-pointed case:
\begin{align*}
d((a\otimes x),(a\otimes y) &=
	\frac{1}{2}d(x,y) \\
d((a\otimes x),(b\otimes y) &=\frac{1}{2} \min \{ d(x,L)+d(T,y)~,~d(x,R)+1+d(R,y) \}
\end{align*}

A formula similar to the last one holds for other cases of $a$, $b$ with $a \neq b$.
\end{lemma}

Given a ($i$-pointed) metric space $~(X_i, d)~$, we use the same letter $~d~$ to identify the metrics on $~M_i\times X_i~$ and $~M_i\otimes X_i$, since it is understood from the context to which $~d~$ we refer.\\

Consider the bi-pointed case. Clearly, by definition, the  metric on $~M_{2}\otimes X_{2}~$  is  one bounded and $~d(l\otimes \bot, r\otimes \top)=1~$. Therefore  $~M_{2}\otimes X_{2}~$ is a bi-pointed metric space and hence $~(M_{2}\otimes X_{2},d)~$  is an object in the category of bi-pointed metric spaces with any choice of morphisms (Lipschitz, short or continuous). One can similarly verify that 
$~M_{3}\otimes X~$ is a tri-pointed metric space with the metric given in Lemma \ref{LemmaMetricOnTensor}.\\


To define $~F_i~$ on the categories $~\mathbf{MS_{i}}^{C}~$ , $~\mathbf{MS_{i}}^{L}~$ and $~\mathbf{MS_{i}}^{S}~$, one needs to show that it acts on morphisms in the expected way. What we need follows from \ref{cnt}, \ref{lip} and \ref{shrt} of the following lemma.

\begin{lemma}
\label{presFunctions}
Let $~X~$ and $~Y~$ be two $i-$pointed metric spaces. If $~f: X \rightarrow Y~$  has any of the following properties then so does $~F_{i}f~$.
\begin{enumerate}[i).]
\item \label{cnt} Continuous
\item \label{lip} Lipschitz
\item \label{shrt} Short map
\item \label{isoemb} Isometric embedding
\end{enumerate}
\end{lemma}

\begin{proof}
(\ref{cnt}). 
Let $~\epsilon>0~$ be arbitrary. Since $~f~$ is continuous, $~\exists\,\delta >0~$ such that $~d(f(x),f(y))<\epsilon~ $,  whenever $d(x,y  )<\delta $.  Choose  $~\delta_{0}=\min \{   \displaystyle\frac {\delta}{2},\displaystyle\frac {1}{4} \}$. Suppose that $~d_{M_i\otimes X}(a\otimes x, b\otimes y)<\delta_{0} ~$. Then we have $~d_{X}(x,L_{X})< \delta$ and $d_{X}(y,T_{X})<\delta $. By the continuity of $~f~$ we have $~d_{Y}(f(x),f(L_{X}))<\epsilon$ and $~d_{Y}(f(y),f(T_{X}))<\epsilon$. Thus, $~d_{M\otimes Y}(F_{i} f (a\otimes x), F_{i}f (b\otimes y) )<\epsilon ~$, which is the required condition for $~F_{i} f~$ to be continuous. This completes the proof.\\\\
(\ref{lip}).	Let $~f~$ be a Lipschitz continuous function with the Lipschitz constant $~k~$. Note that 
\begin{align*}
\text{m} & \text{in} \{ d\left(f(x),L\right)+d(T,f(y))~,~d(f(x),R)+1+d(R,f(y)) \} \\ 
   & =\ \min \{ d(f(x),f(L))+d(f(T),f(y))~,~d(f(x),f(R))+1+d(f(R),f(y)) \}\\
   & \leq  \min \{kd(x,L)+kd(T,y)~,~kd(x,R)+kd(T,L)+kd(R,y)\}\\
   & =  k \, \min \{d(x,L)+d(T,y)~,~d(x,R)+d(T,L)+d(R,y)\}
\end{align*}

Therefore, 
\[
d\left(F_{3} f(a\otimes x),F_{3} f(b\otimes y)\right) = d\left(a\otimes f(x), b\otimes f(y)\right) \leq k d\left(a\otimes x,b\otimes y\right) 
\]
This completes the proof.\\\\
(\ref{shrt}) is proved in \cite{Bhat} and the proof of (\ref{isoemb}) is similar to that of (\ref{lip}).
\end{proof}

One can now define endofunctors $~F_{i}~$ on $~\mathbf{MS_{i}}^{S}~$, $~\mathbf{MS_{i}}^{C}~$ and  $~\mathbf{MS_{i}}^{L}$, as supported by Lemma \ref{presFunctions}.\\     
 
Authors of  \cite{Bhat} have computed the initial algebra $\,(G_{i},h_{i})\,$ and the  final co-algebra $~(S_{i},\psi_{i})\,$ of $\,F_{i}\,$ on $\,\mathbf{MS_{i}}^{S}~$ and they have shown that the final co-algebra is obtained by the Cauchy completion of the initial algebra. It turns out that the initial algebra and the final co-algebra of $\,F_{i}\,$ on $~\mathbf{Set_{i}}~$   are the same as that of $~F_{i}~$ on   $~\mathbf{MS_{i}}^{S}~$ after forgetting the metric structure. Moreover, \cite{Bhat}  exhibits a bi-Lipschitz isomorphism between two  co-algebras of $\,F_{3}\,$ 
on $~\mathbf{MS_{3}}^{L}~$, one being the Sierpinski gasket, and  raises the question whether either of these  co-algebras of $~F_{3}~$ is the final co-algebra for the endofunctor $~F_{3}~$  on $~\mathbf{MS_{3}}^{L}~$. This study was initiated based on this question.\\


 Our contribution is in two directions. In Section \ref{secOnMSC},  we show that the  final co-algebra of  $~F_{i}~$ on $~\mathbf{MS_{i}^{S}}~$ is same as that on $~\mathbf{MS_{i}^{C}}~$. Along the way we show how the mediating morphism from a co-algebra to the final co-algebra can be obtained by the limit of a sequence obtained by iterating the co-algebra. In Section \ref{secOnInitialAlgebraMSC} we show that the initial algebra of $~F_{i}~$ on $~\mathbf{MS_{i}^{S}}~$ is not the initial algebra of $F_i$ on $~\mathbf{MS_{i}^{C}}$. We still do not know whether the initial algebra of $F_i$ on  $~\mathbf{MS_{i}^{C}}$ exists.  Moreover, in Section \ref{secOnMSL}, we show that the final co-algebra  and initial algebra  of $~F_{i}~$  on $~\mathbf{MS_{i}^{S}}~$ are not the final co-algebra and initial algebra of $~F_{i}~$ on $~\mathbf{MS_{i}^{L}}~.$  In the case of $~F_{i}$ on $~\mathbf{MS_{i}^{L}}~$, we do not know whether the final co-algebra and initial algebra exist. However, the results suggests a negative answer to the question that was raised  in \cite{Bhat}, and mentioned in the above paragraph.
 


\section{ Final co-algebra for  $~F_{i}~$ on $~MS_{i}^{C}$} 
\label{secOnMSC}

In this section we consider $~F_{i}~$ defined on $~\mathbf{MS_{i}^{C}}~$. Authors of \cite{Bhat} have computed the initial algebra and final co-algebra of the endofunctor $~F_{i}~$ on  $~\mathbf{MS_{i}^{S}}~$ and shown that the final co-algebra is given by the Cauchy completion of the initial algebra. Let us briefly recall how this is done. Consider the following initial chain starting from the initial object $I$, where all the maps are isometric embeddings.
\[I~\xrightarrow{!}~ M_i\otimes I ~\xrightarrow{M_i\otimes !}~ M_i^{2}\otimes I ~\xrightarrow{M_i^{2}\otimes !} ~ M_i^{3}\otimes I\cdots M_i^{n}\otimes ! ~\xrightarrow{M_i^{n}\otimes !}~ M_i^{n+1}\otimes I \cdots\]

Take $\,C=\bigcup M_i^{n}\otimes I$. Define a relation on $C$ as follows. Let $x,y \in C$. Then$~x\in M_i^{r}\otimes I $ and $y\in M_i^{s}\otimes I ~$  for some $~r,s$. Without lost of generality take $s>r$. The relation is defined by $x \approx y~$ iff $f(x)=y$; where $f=M_i^{s-1}\otimes !\circ \cdots  \circ M_i^{r}\otimes !$. Let $\sim$ be equivalence closure of $\approx$. Take $G=C / \sim$. This $~G~$ is the colimit of the above chain and also the carrier set of the initial algebra. The morphisms arising in the colimit are given by  $C_{n}:M_i^{n}\otimes I \longrightarrow G$, where $C_{n}(x)=[x]$, which are  isometric embeddings. By Adamek's Theorem (see \cite{Adamek}) the initial algebra is given by $(G,g:M_i\otimes G\longrightarrow G)$, where $g:M_i\otimes G\longrightarrow G$ is given by $g(m\otimes [x])=[m\otimes x]$. The carrier set of the final co-algebra is the Cauchy completion of $~G~$ which we denote by $\,S\,$. Throughout this paper we will consider $\,G\,$ as a dense subset of the complete metric space $S$.\\

Authors of \cite{Bhat} have also shown that the initial algebra and final co-algebra of $\,F_{i}\,$ on  $\,\mathbf{Set_{i}}\,$ are the same as that of $\,F_i\,$ on $\,\mathbf{MS_{i}^{S}}\,$, leaving out the metric structure.  One can make use of this fact to compute  the mediating morphism at the set level from a given co-algebra to  the final coalgebra, as the limit of a sequence. We demonstrate it for the tri-pointed case, as the construction for the bi-pointed case is similar. Let $\,(X,e)\,$ be any co-algebra for $\,F_3\,$ on $\mathbf{Set_3}$, where $~e~$ is a  set function. By iterating this coalgebra, we obtain the following chain.
 \[~X \xrightarrow{e}M_3\otimes X\xrightarrow{M_3\otimes e}M_3^{2}\otimes X \xrightarrow{M_3^{2}\otimes e}M_3^{3}\otimes X\cdots M_3^{n-1}\otimes X \xrightarrow{M_3^{n-1}\otimes e}M_3^{n}\otimes X\cdots\]
Set $M_3^0\otimes X=X$ and $M_3^0\otimes e=e$. For an $~x\in X~$, we can iterate this $x$ and obtain the sequence $\left(\chi_n \right)_{n=0,1,\cdots}$ such that $\chi_n = m_{1}\otimes m_{2}\otimes...\otimes m_n  \otimes x_{n} \in M_3^n\otimes X$, given by $\chi_0=x$, $\chi_1=M_3^0\otimes e(\chi_0) = e(x)$ and $\chi_n$ is given inductively by  $\chi_n = \ M_3^{n-1} \otimes e \left ( \chi_{n-1}\right) =\ M_3^{n-1} \otimes e \left (m_{1}\otimes m_{2}\otimes...\otimes m_{n-1}  \otimes x_{n-1}\right)$. Then we have a corresponding sequence  $~(\theta_{n})~$ in $~G~$ as follows. $\theta_{1}=[m_{1}\otimes \overline{x_{1}}]\, , \, \theta_{2}=[m_{1}\otimes m_{2}\otimes  \overline{x_{2}}]\, , \, \theta_{3}=[m_{1}\otimes m_{2}\otimes m_{3} \otimes \overline{x_{3}}],\cdots  $; where $~\overline{x_{i}}~$ is chosen to be $~T~$ or $~L~$ or $~R~$ provided that the value of $~m_{i}~$ is $~a~$  or $~b~$  or $~c~$ respectively. Since $~G~$ is also the carrier set arising in the initial algebra of  $~F_{i}~$ on $~\mathbf{MS_{i}^{S}}~$, it has a metric structure.  Now we shall prove that the sequence $~(\theta_{n})~$ is a Cauchy sequence in $~G~$.\\
Let $\epsilon>0\,$ and choose $\,N\in\mathbf{N}$ such that $\,N>\displaystyle\frac{\ln \frac{1}{\epsilon}}{\ln 2}\,$. Now for $\,q>p>N$, consider $d_{G}(\theta_{p},\theta_{q}  )=d_{G}([m_{1}\otimes \cdots m_{p} \otimes \overline{x_{p}}],[m_{1}\otimes  \cdots \otimes m_{q} \otimes \overline{x_{q}}]  )$. The right side of this equality is equal to  $~d_{G}([m_{1}\otimes \cdots m_{p}\otimes l_{p+1}\otimes\cdots  \otimes l_{q}\otimes \overline{y_{q}}],[m_{1}\otimes  \cdots \otimes m_{q} \otimes \overline{x_{q}}]  )~$; where $~m_{1}\otimes \cdots m_{p}\otimes l_{p+1}\otimes\cdots  \otimes l_{q}\otimes \overline{y_{q}}~\sim~ m_{1}\otimes  \cdots \otimes m_{q} \otimes \overline{x_{q}}~$. As $C_{q}$ defined above is an isometric embedding ,$~d_{G}([m_{1}\otimes \cdots m_{p}\otimes l_{p+1}\otimes\cdots  \otimes l_{q}\otimes \overline{y_{q}}],[m_{1}\otimes  \cdots \otimes m_{q} \otimes \overline{x_{q}}]  )$ is equal to $~d_{M^q\otimes I}(m_{1}\otimes \cdots m_{p}\otimes l_{p+1}\otimes\cdots  \otimes l_{q}\otimes \overline{y_{q}},m_{1}\otimes  \cdots \otimes m_{q} \otimes \overline{x_{q}}  )$ which is less than or equal to $~\frac{1}{2^p}~$ (see Lemma 15 of \cite{Bhat}). But  $~\frac{1}{2^p}<\frac{1}{2^N}<\epsilon~$. Thus we have $d_{G}(\theta_{p},\theta_{q}  )<\epsilon$ which is the required condition for $(\theta_{n})$ to be a Cauchy sequence. Thus one can consider the $\displaystyle\lim_{n\rightarrow \infty}\theta_{n}~$ in $S$. 

\begin{lemma}
$\displaystyle\lim_{n\rightarrow \infty}\theta_{n}~$ is independent of the choice of the sequence $\left(\chi_n \right)$.
\end{lemma}

\begin{proof}
Suppose we consider two choices for $\chi_n$, $\chi_n=m_{1}\otimes m_{2}\otimes\cdots\otimes m_n  \otimes x_{n}=m_{1}^{'}\otimes m_{2}^{'}\otimes\cdots\otimes m_{n}^{'}  \otimes y_{n}$. Thus initially we have  $~m_1\otimes x_1=m_{1}'\otimes y_{1}~$. Without loss of generality take $m_{1}=a$. Then there are two possibilities for $m_{1}^{'}$ to be, namely $b$ or $c$, and we have $a\otimes L=b\otimes T$ or $a\otimes R=c\otimes T$ respectively. Consider the case $~a\otimes L=b\otimes T~$. Then the corresponding sequences $~\theta_{n}\,, \,\theta_{n}^{'}\in G~$  given by $~\theta_1=[a\otimes L]\,,\,\theta_2=[a\otimes b\otimes L]\,,\,\theta_{n}=[a\otimes b\otimes \cdots b\otimes L],\cdots$ and  $~\theta_{1}^{'}=[b\otimes T]\,,\,\theta_2^{'}=[b\otimes a\otimes T]\,,\,\theta_{n}^{'}=[b\otimes a\otimes \cdots a\otimes T],\cdots$ are equal. Similarly we can show $~\theta_n=\theta_{n}^{'}~$ for the other cases. Thus their limits are the same. 
\end{proof}

Thus one can define a function $\,f:X\rightarrow S\,$ by $~f(x)=\displaystyle\lim_{n\rightarrow \infty}\theta_{n}~$, which is well defined according to the above lemma.

\begin{proposition} 
The mediating morphism $~f~$  for a given co-algebra is  given by $~f(x)=\displaystyle\lim_{k\rightarrow \infty}\theta_{k}~$. 
\end{proposition} 

\begin{proof}
We only need to prove that $\,\left( M_i\otimes f \right) \circ e=\psi \circ f\,$ to show that the following diagram commutes, where  $\,\psi:S\longrightarrow M\otimes S\,$ is given by $\psi([m_{1}\otimes\cdots\otimes m_k  \otimes \overline{x_{k}}])=m_{1}\otimes [m_{2}\otimes\cdots\otimes m_k  \otimes \overline{x_{k}}]$. Notice that $\,S\,$ is the Cauchy completion of the $\,G\,$ and we consider $\,G\,$ as the dense subset of $\,S\,$. Thus $\psi$ found in \cite{Bhat} is written above form.

$$\begin{array}{ccc}
X & \underrightarrow{~~~~ e ~~~~} &  M_i\otimes X\\\\
\downarrow{ f} &  & \downarrow{  M_i\otimes f}\\\\
S & \overrightarrow{~~~~ \psi~~~~} & M_i\otimes S \\
\end{array}$$
For $\,x\in X\,$ and $e(x)=m_{1}\otimes x_{1}\,,\, \left( M_i\otimes e\right) (m_{1}\otimes x_{1})=m_{1}\otimes m_{2}\otimes x_{2}, \cdots $. Then we have $f(x)=\displaystyle\lim_{k\rightarrow \infty}[m_{1}\otimes\cdots\otimes m_k  \otimes \overline{x_{k}}]$ and $f(x_{1})=\displaystyle\lim_{k\rightarrow \infty}[ m_{2}\otimes\cdots\otimes m_k  \otimes \overline{x_{k}}]$. Now $\psi \left(f(x)\right)=\psi \left(\displaystyle\lim_{k\rightarrow \infty}[m_{1}\otimes\cdots\otimes m_k  \otimes \overline{x_{k}}]\right)=\displaystyle\lim_{k\rightarrow \infty}\psi([m_{1}\otimes\cdots \otimes \overline{x_{k}}])=\displaystyle\lim_{k\rightarrow \infty}\left( m_{1}\otimes [m_{2}\otimes\cdots\otimes m_k  \otimes \overline{x_{k}}]\right)$. Now  $M_i\otimes f e(x)=M_i\otimes f(m_{1}\otimes x_{1})=m_{1}\otimes f(x_{1})~$ which is equal to $~m_{1}\otimes \displaystyle\lim_{k\rightarrow \infty}[ m_{2}\otimes\cdots\otimes m_k  \otimes \overline{x_{k}}]$. But $m_{1}\otimes \displaystyle\lim_{k\rightarrow \infty}\theta_{k}=\displaystyle\lim_{k\rightarrow \infty}m_{1}\otimes\theta_{k}$. Thus the above diagram commutes.  
\end{proof} 

Now suppose that $(X, e)$ is an $~F_i~$ co-algebra on $~\mathbf{MS_{i}^{C}}~$. Then leaving out the metric structure, we can calculate the mediating morphism at the set level as given above. The following lemma states that this mediating morphism $~f~$ is continuous in the metric setting for both the bi-pointed case and the tri-pointed case. We prove only the tri-pointed case as the bi-pointed case is similar.  

\begin{proposition}
Let $\,I\,$ be the initial object and $\,X\,$ be any object in the category $\textbf{MS}_{i}^{C}$. Let $\mu$ be the  unique morphism $\mu:I\longrightarrow X$(by initiality condition). Then $\mu$ and all $ M_{i}^{n}\otimes \mu$'s are isometric embeddings.   
\end{proposition}
Proving $\mu$ is an isometric embedding is straightforward. The other part follows from Lemma \ref{presFunctions}.
\begin{lemma}
	The mediating morphism $~f~$, defined above, is continuous.
\end{lemma} 
\begin{proof}
Let $x \in X$ and $\,\epsilon >0\,$ be arbitrary. From the definition of $f$ we have $\displaystyle d_{S}(f(x),f(y)) = \lim_{p\rightarrow \infty}d_{G}\{[m_{1}\otimes \cdots m_{p}\otimes \overline{x_{p}}],[n_{1}\otimes \cdots n_{p}\otimes \overline{y_{p}}]\}$. Therefore, there is some  $N\in \mathbb{N}$ such that 
$$d_{S}(f(x),f(y))	-\frac{\epsilon}{4} < d_{G}\{[m_{1}\otimes \cdots m_{p}\otimes \overline{x_{p}}],[n_{1}\otimes \cdots n_{p}\otimes \overline{y_{p}}]\}$$ 
for all $p>N$ and for all $y\in X$. Choose  $p=\max\left\{N+1, \displaystyle\frac{\ln\frac{4}{\epsilon}}{\ln 2}\right\} $.  Denote $\,g_{p}=M_i^{p-1}\otimes e\circ \cdots \circ M_i\otimes e\circ e$. Because $g_p$ is continuous at $x$, $\exists \delta_{p} >0$ such that 
 $~d_{M_i^{p}\otimes I }(g_{p}(x),g_{p}(y))<\frac{\epsilon}{4}$ whenever $d(x,y)<\delta_{p}$. Suppose that $d(x,y)<\delta_p$. Then we have

 \begin{eqnarray*}
d_{S}(f(x),f(y))	
&<&\frac{\epsilon}{4} + d_{G}\{[m_{1}\otimes \cdots m_{p}\otimes \overline{x_{p}}],[n_{1}\otimes \cdots n_{p}\otimes \overline{y_{p}}]\} \\
&=&\frac{\epsilon}{4}+d_{M_i^{p}\otimes I}\{m_{1}\otimes \cdots m_{p}\otimes \overline{x_{p}},n_{1}\otimes \cdots n_{p}\otimes \overline{y_{p}}\}\\ &~&~~~~~~~~~~~\lceil~\because \, M_{i}^{p}\otimes I\xrightarrow{~C_p~}G~ \textrm{isometric embedding}~\rfloor\\
&=&\frac{\epsilon}{4}+d_{M_i^{p}\otimes X}\{m_{1}\otimes \cdots m_{p}\otimes \overline{x_{p}},n_{1}\otimes \cdots n_{p}\otimes \overline{y_{p}}\}\\ &~&~~~~~~~~~~~\lceil~\because \, M_{i}^{p}\otimes I\xrightarrow{~M_{i}^{p}\otimes \mu~} M_{i}^{p}\otimes X~\textrm{isometric embedding}~\rfloor\\
 &\leq&\frac{\epsilon}{4}+ d_{M_i^{p}\otimes X}\{m_{1}\otimes \cdots m_{p}\otimes \overline{x_{p}},m_{1}\otimes \cdots m_{p}\otimes x_{p}\}+\\
&~&\,\,\,\,\,d_{M_i^{p}\otimes X}\{m_{1}\otimes \cdots m_{p}\otimes x_{p},n_{1}\otimes \cdots n_{p}\otimes y_{p}\}+\\
&~&\,\,\,\,\,d_{M_i^{p}\otimes X}\{n_{1}\otimes \cdots n_{p}\otimes y_{p},n_{1}\otimes \cdots n_{p}\otimes  \overline{y_{p}}\}\\
&\leq& \frac{\epsilon}{4}+\frac{1}{2^{p}}+d_{M^{p}\otimes X}\{g_{p}(x),g_{p}(y)\}+\frac{1}{2^{p}}\\ 
&<& \frac{\epsilon}{4}+\frac{\epsilon}{4}+\frac{\epsilon}{4}+\frac{\epsilon}{4}=\epsilon
\end{eqnarray*}
completing the proof.	 
\end{proof}

The mediating morphism is uniquely determined for a given co-algebra $\,(X,e)\,$ and hence it is unique. Therefore,  $~(S,\psi)~$ is the final co-algebra of $~F_{i}~$ on  $~\mathbf{MS_{i}}^{C}~$.

\begin{proposition}
The final co-algebra of $~F_{i}~$ on  $~\mathbf{MS_{i}}^{C}~$ is the same as that of $F_i$ on $~\mathbf{MS_{i}}^{S}~$. 
\end{proposition}


\section{The initial algebra of $~F_{i}~$ on $~MS_{i}^{C}~$ is not that of  $~F_{i}~$ on $~MS_{i}^{S}~$} 
\label{secOnInitialAlgebraMSC}

  As the final co-algebra of $~F_{i}~$ on  $~\mathbf{MS_{i}}^{C}~$ is the same as the final co-algebra of $~F_{i}~$ on  $~\mathbf{MS_{i}}^{S}~$, one may wonder whether a similar result holds for the initial algebra of $~F_{i}~$ on $~\mathbf{MS_{i}}^{S}~$ and  $~\mathbf{MS_{i}}^{C}~$ . We  give a negative answer to this question by giving a counter example (Example \ref{ExampleInitialAlgeMSC}). First we need some preliminary results. The initial algebra of $~F_{i}~$  on  $~\mathbf{MS_{i}}^{S}~$  is the same as the initial algebra of $~F_{i}~$  on  $~\mathbf{Set_{i}}~$ (see\cite{Bhat}). Lemma \ref{mediationMorAlgebra} given below states a way to compute the mediating morphism at the set level for both the tri-pointed and bi-pointed cases. We will use this later to decide whether the mediating morphism is continuous or Lipschitz. Again, we demonstrate it only for the tri-pointed case, as the bi-pointed case is similar.   
  
Let $~(X,e)~$ be an algebra for $~F_{3}~$ on $\mathbf{Set_3}$, where $\,e\,$ is a set function. Let $~(G,g)~$ be the initial algebra of $~F_{3}~$ in this setting (See \cite{Bhat}). For $~x=[(m_{1}\otimes \cdots m_{k}\otimes d)]\in G~$; where $~d\in \{T,L,R\}$, $~\overline{x}=m_{1}\otimes \cdots m_{k}\otimes d_{X}~$. $~d_{X}~$ is chosen corresponding to  $~d~$. For instance if $~d=T~$, then $~d_{X}=T_{X}~$. Now consider the chain $~M_3^{k}\otimes X\xrightarrow{M_3^{k-1}\otimes e}M_3^{k-1}\otimes X \cdots M_3\otimes X\xrightarrow{e} X~$. Take $~g_{k}=e \circ M_3\otimes e \circ \cdots M_3^{k-1}\otimes e~$,  and define $~f:G \longrightarrow X ~$ by $~f(x)=g_{k}(\overline{x})~$.
 \begin{lemma}
 \label{mediationMorAlgebra}
The mediating morphism for a given algebra $\,(X,e)\,$ for $F_i$ on $\mathbf{Set_i}$, $i=2$ or $3$, is given by $f(x)=g_{k}(\overline{x})~$, where $g_{k}=e \circ M_i\otimes e \circ \cdots M_i^{k-1}\otimes e$. 
\end{lemma} 
\begin{proof}
Let us first prove that $~f~$ is well defined. Let $~[(p_{1}\otimes\cdots  p_{r}\otimes d_{r})]=[(q_{1}\otimes\cdots  q_{s}\otimes d_{s})]$. Thus we have $~p_{1}\otimes\cdots  p_{r}\otimes d_{r}\sim q_{1}\otimes\cdots  q_{s}\otimes d_{s}~ $. Without loss of generality take  $~s>r~$ and consider the following initial chain, where $0$ is the initial object. 
$$0\xrightarrow{!}M_i\otimes 0\xrightarrow{M_i\otimes !}\cdots M_i^{r}\otimes 0\xrightarrow{M_i^{r}\otimes !}M_i^{r+1}\otimes 0\cdots \xrightarrow{M_i^{s-1}\otimes !}M_i^{s}\otimes 0$$
Thus $\left( M_i^{s-1}\otimes !\circ M_i^{s-2}\otimes !\cdots M_i^{r+1}\otimes !\circ M_i^{r}\otimes ! \right) \left( p_{1}\otimes\cdots  p_{r}\otimes d_{r} \right)=  q_{1}\cdots  \otimes d_{s}$\\
Since $~0~$ is the initial object, the following diagram commutes.  
$$\begin{array}{ccc}
M_i\otimes X & \underrightarrow{~~~~ e ~~~~} &  X\\\\
\uparrow{M_i\otimes  \eta} &  &\uparrow{\eta}  \\\\
M_i\otimes 0 & \overleftarrow{~~~~~ !~~~~~} & 0 \\
\end{array}$$
and hence by applying $M_i\otimes -$ repeatedly we have the following commuting square.  
$$\begin{array}{ccc}
M_i^{s}\otimes X & \underrightarrow{~ M_i^{r}\otimes e \circ M_i^{r+1}\otimes e \circ \cdots \circ M_i^{s-2}\otimes e \circ M_i^{s-1}\otimes e~} &  M_i^{r}\otimes X\\\\
\uparrow{M_i^{s}\otimes  \eta} &  &\uparrow{M_i^{r}\otimes  \eta}   \\\\
M_i^{s}\otimes 0 & \overleftarrow{~~~~~~M_i^{s-1}\otimes !\circ M_i^{s-2}\otimes !\cdots M_i^{r+1}\otimes !\circ M_i^{r}\otimes ! ~~~~~~} & M_i^{r}\otimes 0 \\
\end{array}$$

Thus $$~M_i^{r}\otimes e \circ M_i^{r+1}\otimes e \circ \cdots \circ M_i^{s-2}\otimes e \circ M_i^{s-1}\otimes e(q_{1}\otimes\cdots  q_{s}\otimes d_{s_{X}})=p_{1}\otimes\cdots  p_{r}\otimes d_{r_{X}}$$
Now consider $~f[(q_{1}\otimes\cdots  q_{s}\otimes d_{s})]= e\circ M_i\otimes e \cdots\circ M_i^{s-1}\otimes e (q_{1}\otimes\cdots  q_{s}\otimes d_{s_{X}})$. The right side of this equation is equal to  $$  e\circ M_i\otimes \cdots \otimes e\circ M_i^{r-1}\otimes e (M_i^{r}\otimes e  \circ \cdots  \circ M_i^{s-1}\otimes e(q_{1}\otimes\cdots  q_{s}\otimes d_{s_{X}})) $$ which is equal to $\,
 e\circ M_i\otimes e\circ \cdots \circ M_i^{r-1}\otimes e(p_{1}\otimes\cdots  p_{r}\otimes d_{r_{X}} )=f[(p_{1}\otimes\cdots  p_{r}\otimes d_{r})]
$. Thus $f[(q_{1}\otimes\cdots  q_{r}\otimes d_{s})]=f[(p_{1}\otimes\cdots  p_{r}\otimes d_{r})]
$
 and therefore $~f~$ is well defined.\\
 
We are left to show that the following  diagram commutes. 

$$\begin{array}{ccc}
M_i\otimes G & \underrightarrow{~~~~ g ~~~~} &  G\\\\
\downarrow{M_i\otimes f} &  & \downarrow{  f}\\\\
M_i\otimes X & \overrightarrow{~~~~ e~~~~} & X \\
\end{array}$$

For any  $~m_{0}\otimes [(m_{1}\otimes m_{2}\cdots\otimes m_{k}\otimes x_{k})]\in M_i\otimes G~$ and $~e(m_{k}\otimes x_{k})=x_{k-1}$; where $k=1,2,\cdots$, Consider the following equality.
\[
f[(m_{0}\otimes m_{1}\cdots  \otimes m_{k}\otimes x_{k})]= e \circ M_i\otimes e \cdots M_i^{k}\otimes e (m_{0}\otimes m_{1}\cdots  \otimes m_{k}\otimes x_{k_{X}})
\]

Applying $~M_i^{k}\otimes e~$ to the element $m_{0}\otimes m_{1}\cdots  \otimes m_{k}\otimes x_{k_{X}}$,  the right side becomes $~ e\circ  \cdots\circ M_i^{k-1}\otimes e (m_{0}\otimes m_{1}\cdots  \otimes m_{k-1}\otimes x_{k-1})$. Continuing  this process , eventually we get    $f[(m_{0}\otimes m_{1}\cdots  \otimes m_{k}\otimes x_{k})]=  e(m_{0}\otimes x_{0}) 
$. Similarly $~f[( m_{1}\cdots  \otimes m_{k}\otimes x_{k})]=x_{0}$. Thus we have $e\circ M_i\otimes f \{m_{0}\otimes[( m_{1}\cdots  \otimes m_{k}\otimes x_{k})] \}=f\circ h\{m_{0}\otimes[( m_{1}\cdots  \otimes m_{k}\otimes x_{k})] \}=e(m_{0}\otimes x_{0})$ and hence 
$f\circ h=e \circ M_i\otimes f$
\end{proof}

\begin{proposition}[See \cite{Bhat}]
 Let $~X_{0}=\{\bot,\top\}~$ be a bi-pointed set . Then there are isometries   $c_{0},c_{1},c_{2}\cdots,c_{n},\cdots$ such that $\,M_2^{n}\otimes X_{0}\cong D_{n}\,;~\forall\,n\in \mathbf{Z^{+}}\cup\{0\}\,$. Here $\,D_{q}=\{\displaystyle\frac{p}{2^{q}}~/~0\leq p \leq 2^{q}~\&\,p,q\in \mathbf{Z^{+}}\cup\{0\}\}\,$ and $c_{0}:X_0\longrightarrow D_0=\{0,1\}$ is given by $c_{0}(\bot)=0\,,\,c_{0}(\top)=1$. Moreover, $c_k:M_{2}^{k}\otimes X_{0}\longrightarrow D_{k} $ is given inductively by

 $$c_{k}(m_{1}\otimes \cdots m_{k}\otimes d)=\left\{%
	\begin{array}{ll}
	\displaystyle\frac{1}{2}c_{k}(m_{2}\otimes \cdots m_{k}\otimes d), & \hbox{$m_1=l$;} \\
	\displaystyle\frac{1}{2}(c_{k}(m_{2}\otimes \cdots m_{k}\otimes d)+1), & \hbox{$m_{1}=r$.}
	\end{array}%
	\right.$$
  \end{proposition}
  
The following example states that $~(D, \phi )~$ is not an initial algebra of $~F_{2}~$ on  $~\mathbf{MS_{2}}^{C}~$ and $~\mathbf{MS_{2}}^{L}$. We will show it for the continuous case. The same example will works for the Lipschitz case too.    

\begin{example}
\label{ExampleInitialAlgeMSC}
\emph{
Let $~X_{0}=\{\bot,\top\}~$ be a bi-pointed set .   Consider the function $~e:M_2\otimes X_{0} \longrightarrow  X_{0}~$ given by $~e(l\otimes 0)=0,e(l\otimes 1)=0~ $ and $~e(r\otimes 1)=1~$. Clearly this $~e~$ is Lipschitz as it is a function from a finite metric space and it is also continuous. The initial algebra of $~F_{2}~$ on $~\mathbf{MS_{2}}^{S}~$ is  the pair $~(D, \phi)~$: where  $~D~$ is the dyadic rationals in the unit interval and $~\phi:M_2\otimes D \longrightarrow D~$ is as follows. }
\end{example}
$$\phi(m\otimes x)=\left\{%
	\begin{array}{ll}
	\displaystyle\frac{x}{2}, & \hbox{$m=l$;} \\\\
	\displaystyle\frac{x+1}{2}, & \hbox{$m=r$.}
	\end{array}%
	\right.$$
Suppose that $~(\phi, D)~$ is the initial algebra of $~F_{2}~$ on $~\mathbf{MS_{2}}^{C}$. Then there is a unique continuous function $\,f\,$ such that the expected diagram commutes.  
Suppose that $\,f\,$ is continuous. Then $\,\exists\, \delta>0\,$ such that $~d(f(x),f(y))<\displaystyle\frac{1}{2}$   whenever $~\forall \, d(x,y)<\delta$. Choose $\,n\,$ large enough so that $~\displaystyle\frac{1}{2^{n}}<\delta~$. Thus $~d(1,\displaystyle\frac{2^{n}-1}{2^{n}})<\delta~$ and $~d\left(f(1),f(\displaystyle\frac{2^{n}-1}{2^{n}})\right)<\displaystyle\frac{1}{2}~$. Now $\,f(1)=1\,$ as  $\,f\,$ preserves the distinguished elements. Now $\,\frac{2^{n}-1}{2^{n}}\in D_{n}\,$ and from the straightforward computation  we have $\,c_{n}(r\otimes r \otimes\cdots \otimes r\otimes 0)=\frac{2^{n}-1}{2^{n}}\,$. Thus the element $\,\frac{2^{n}-1}{2^{n}}\,$ is identified with $\,r\otimes r \otimes\cdots \otimes r\otimes 0\,$.  Thus  $\,f(\displaystyle\frac{2^{n}-1}{2^{n}})\,$ can be evaluated as follows.
\begin{eqnarray*}
f(\displaystyle\frac{2^{n}-1}{2^{n}})&=&e_{2}\circ M_{2}\otimes e_{2} \circ \cdots M_{2}^{n-2}\otimes e_{2}\circ M_{2}^{n-1}\otimes e_{2}((r\otimes r \otimes\cdots \otimes r\otimes 0)\\
&=&e_{2}\circ M_{2}\otimes e_{2} \circ \cdots M_{2}^{n-2}\otimes e_{2}(( r \otimes\cdots \otimes r\otimes 0)\\
&=&\colon~~~\colon~~~~\colon  \\
&=&e_{2}(r\otimes 0)=0\\
\end{eqnarray*}
Therefore we have $d\left(f(1),f(\displaystyle\frac{2^{n}-1}{2^{n}})\right)=d(1,0)=1<\displaystyle\frac{1}{2}~$
which is obviously a contradiction.  Thus the  mediating morphism $~f~$ is  not continuous.  Hence $~(\phi,D)~$ is not the initial algebra for $~F_{2}~$ on $\mathbf{MS_{2}}^{C}$.

\begin{example}
\label{ExampleInitialAlgeMSCtri}
\emph{
Consider the tri-pointed set $~Y_{0}=\{T,L,R\}~$ and the function $~e_3:M_3\otimes Y_{0} \longrightarrow  Y_{0}~$ given by $~e_3(a\otimes T)=T\,,\,e_3(c\otimes R)=R~ $ and $~e_3(a\otimes L)=e_{3}(a\otimes R)=e_{3}(b\otimes L)=e_{3}(b\otimes R)=L~$. Because $~e_3~$ is a function from a finite metric space, it is Lipschitz  and hence  continuous too. Consider the initial algebra $~(G, g)~$ of $~F_{3}~$ on $~\mathbf{MS_{3}}^{S}~$, where  $~g:M_3\otimes G \longrightarrow G~$ is given by $g(m\otimes [x])=[m\otimes x]$. See the beginning of Section 3}
\end{example}

Leaving out the metric structure, $(G,g)$ is also the  initial algebra of $F_3$ on $\textbf{Set}_3$ and $( Y_0,e_{3})$ is a $F_3$ algebra on $\textbf{Set}_3$.Thus there exists a unique $f:G\longrightarrow Y_0$ such that following diagram commutes. 

$$\begin{array}{ccc}
M_3\otimes G & \underrightarrow{~~~~ g ~~~~} &  G\\
\downarrow{M_3\otimes f} &  & \downarrow{  f}\\
M_3\otimes Y_0 & \overrightarrow{~~~~ e_3~~~~} & Y_0 \\
\end{array}$$

The map $\,f\,$ was found explicitly in Lemma \ref{mediationMorAlgebra}. Suppose that $\,f\,$ is continuous at $[T]\in G$. Then $\,\exists\, \delta>0\,$ such that $~d(f([T]),f(y))<\displaystyle\frac{1}{2}$ , whenever $d(T,y)<\delta$. Choose $n$ large enough so that $~\displaystyle\frac{1}{2^{n}}<\delta~$. Note that $[T]=[a\otimes T]=[a\otimes a\otimes T]=\cdots=[a\otimes a\otimes \cdots \otimes a \otimes T]=\cdots,$  and $f([T])=f ([a\otimes a\otimes \cdots \otimes a \otimes T])=T_{Y_0}$ as the distinguished elements are preserved by $f$. Let $y=[a\otimes a\otimes \cdots \otimes a \otimes L]$, where $a$ occurs $(n+1)$ times. Then $d([T],y)=\frac{1}{2^{n+1}}<\frac{1}{2^{n}}<\delta$. However,  
\begin{eqnarray*}
f(y_0)&=&f ([a\otimes a\otimes \cdots \otimes a \otimes L])\\
&=&e_{3}\circ M_{3}\otimes e_{3}\circ \cdots M_{3}^{n-1}\otimes e_{3}\circ M_{3}^{n}\otimes e_{3}(a\otimes a\otimes \cdots \otimes a \otimes L)\\
&=&e_{3}\circ M_{3}\otimes e_{3}\circ \cdots M_{3}^{n-1}\otimes e_{3}(a\otimes \cdots \otimes a \otimes L)\\
&=&\colon~~~\colon~~~~\colon  \\
&=&e_{3}(a\otimes L)=L_{Y_0}.\\
\end{eqnarray*}
Therefore $\,d(f(T),f(y_0))=d(T_{Y_0},L_{Y_0})=1<\frac{1}{2}\,$, which is  a contradiction. Thus $f$ is not continuous and hence $f$ is not Lipschitz. Hence $(G,g)$ is not the initial algebra of $F_3$ on  $~\mathbf{MS_{3}}^{C}~$ as well as on $~\mathbf{MS_{3}}^{L}~$.

\section{Final co-algebra and initial algebra  of $~F_{i}~$ on $~MS_{i}^{S}~$ are not that of  $~F_{i}~$ on $~MS_{i}^{L}~$} 
\label{secOnMSL}
    In this section we answer two questions. One is  the question raised in \cite{Bhat}, whether $~(S_{i},\psi_{i})~$, the final co-algebra of $~F_{i}~$ on $~\mathbf{MS_{i}^{S}}$ is the final co-algebra of $~F_{i}~$ on $~\mathbf{MS_{i}^{L}}.~$ In Section \ref{secOnMSC} we have shown that $~(S_{i},\psi_{i})~$  is the final co-algebra of $~F_{i}~$ on $~\mathbf{MS_{i}^{C}}~$. However, we provide a negative answer to the question by showing that $~(S_{i},\psi_{i})~$ is not the final co-algebra of $~F_{i}~$ on $~\mathbf{MS_{i}^{L}}$.  The initial algebra $~(G_{i}, g_{i})~$  of $~F_{i}~$ on $~\mathbf{MS_{i}^{S}}~$, after leaving out the metric structure, is the same as that of $F_i$ on $~\mathbf{Set_{i}}~$ (See \cite{Bhat}). One may ask a similar question, whether  $\,(G_{i},g_{i})\,$  of $\,F_{i}\,$ on $\,\mathbf{MS_{i}^{L}}\,$, after leaving out the metric structure, is the same as that of $\,F_i\,$ on $\,\mathbf{Set_{i}}\,$. We give a negative answer to this question too.\\ 
 
As a consequence of Lemma \ref{presFunctions}, we have $~F_{i}~ $ as an endofunctor on  $~\mathbf{MS_{i}^{L}}~$. Recall the  final co-algebra of $~F_{2}~$ on  $~\mathbf{MS_{2}^{S}}~$ which is  $~(I,i)~$; where $~I=[0~,~1]~$ and $i:I\longrightarrow M_{2}\otimes I$ is given by  
\[
i(x)=\left\{%
\begin{array}{ll}
l\otimes x, & \hbox{$x\in [0~\frac{1}{2}] $;} \\
r\otimes x, & \hbox{$x\in [\frac{1}{2}~1]$.}
\end{array}%
\right.
\]

\begin{example}\label{ExampleFinalCoalgebra}
\emph{
Define $e:I\longrightarrow M_{2}\otimes I$  by }

$$e(x)=\left\{%
\begin{array}{llll}
l\otimes 0, & \hbox{$x\in [0~\frac{1}{4}] $;} \\\\
l\otimes (4x-1), & \hbox{$x\in [\frac{1}{4}~\frac{1}{2}] $;}\\\\
r\otimes (4x-2), & \hbox{$x\in [\frac{1}{2}~\frac{3}{4}] $;}\\\\
r\otimes 1, & \hbox{$x\in [\frac{3}{4}~1]$.}
\end{array}%
\right.$$
\end{example}

One can easily  show that $~e~$ is  Lipschitz with Lipschitz constant 2 and thus  $~(e,I)~$ is a co-algebra in $~\mathbf{MS_{i}^{L}}\,$. Since $~(I,i)~$ is the final co-algebra in $\mathbf{Set_2}$, after forgetting the metric structure, there exists a unique set function $~f:I\longrightarrow I~$ such that the expected diagram commutes.

$$\begin{array}{ccc}
I & \xleftarrow{~~~~ i^{-1} ~~~~} &  M_{2}\otimes I\\\\
\uparrow{f} &  & \uparrow{M_{2}\otimes  f}\\\\
I & \overrightarrow{~~~~ e~~~~} & M_{2}\otimes I
\end{array}$$

By commutativity $~f~$ must satisfy the following conditions.

\[
f(x)=\left\{%
\begin{array}{ll}
\displaystyle 0  & \hbox{$x\in [0~,~\frac{1}{2}] $;} \\\\
\displaystyle\frac{f(4x-1)}{2} & \hbox{$x\in [\frac{1}{4}~,~\frac{1}{2}] $;} \\\\
\displaystyle\frac{1+f(4x-2)}{2} & \hbox{$x\in [\frac{3}{4}~,~1] $;} \\\\
\displaystyle 1  & \hbox{$x\in [\frac{3}{4}~,~1] $.} \\
\end{array}%
\right.
\]

Define the following families of intervals for $n =1, 2, 3, \cdots$. 
\begin{align*}
I_n & = \left[\displaystyle \frac{1}{4}~,~\displaystyle\frac{1}{4}+\cdots+\frac{1}{4^{n}}\right]\\ 
J_n & = \left[\displaystyle \frac{1}{4}+\cdots+\frac{1}{4^{n}}+\frac{3}{4^{n+1}}~,~\displaystyle\frac{1}{4}+\cdots+\frac{1}{4^{n}}+\frac{4}{4^{n+1}}\right]
\end{align*}

Using the conditions the mediating morphism satisfies,  we will show that $~f~$ satisfies the following properties.

\begin{enumerate}[(a)]
	\item $f(x)=0~, \,\,\, \forall~n\in\mathbb{N} \text{ and } \forall x\in I_n ~$ \label{cond1}
    \item $~f(x)=\displaystyle \frac{1}{2^{n}}~, \,\,\, \forall~n\in\mathbb{N} \text{ and } \forall x\in J_n$ \label{cond2}
\end{enumerate} 
We shall prove these properties by induction on $~n~$. First let us prove (\ref{cond1}).  For $n=2$ and $~x\in I_2$, we have $~~4x-1\in [0~,~\frac{1}{4}]~$ and $~f(4x-1)=0$. Thus  $~f(x)=\frac{f(4x-1)}{2}=0$. Suppose that $f(x)=0~,~\forall x\in I_n~$. Let $~x\in I_{n+1}$. Then $~~4x-1\in I_n~$ and $~f(4x-1)=0$. Thus $~~f(x)=\frac{f(4x-1)}{2}=0$.  Thus by induction $~f(x)=0~,$ for $ x\in I_n$.\\
 
To prove (\ref{cond2}), first consider the case $~n=1$ and let $x\in J_1$. We then have $\,4x-1\in [\dfrac{3}{4}\,,\,1] $ and $~f(4x-1)=1\,$. Thus $\,f(x)=\dfrac{f(4x-1)}{2}=\displaystyle \dfrac{1}{2}$. Now suppose that for any $x\in J_n$,  $~f(x)=\displaystyle\frac{1}{2^{n-1}}$. Let $~x\in J_{n+1}~ $. We have $~4x-1\in J_n $ and $f(4x-1)=\dfrac{1}{2^n}~$. Thus $~f(x)=\frac{f(4x-1)}{2}=\displaystyle \frac{1}{2^{n+1}}$. Thus by induction $~f(x)=\displaystyle\frac{1}{2^{n}}~,x\in J_n$.

With (\ref{cond1}) and (\ref{cond2}) being proved, to show that $~f~$ is not Lipschitz, suppose to the contrary that $f$ is Lipschitz. Then we have some $k>0$ such that $~d(e(x),e(y))\leq k d(x,y),\,\,~\forall\, x,y \in I$.  Choose $~x= \displaystyle\frac{1}{4}+\cdots+\frac{1}{4^{n}}+\frac{1}{4^{n+1}}~$ and $~y=\displaystyle\frac{1}{4}+\cdots+\frac{1}{4^{n}}+\frac{3}{4^{n+1}}$. Then $~f(x)=0~$ and $~f(y)=\displaystyle \frac{1}{2^{n}}$. From the Lipschitz condition, we have $~\displaystyle \frac{1}{2^{n}}\leq k\displaystyle \frac{2}{4^{n+1}}, ~\forall\, n\in \mathbb{N}~$; which implies 
 $~k\geq 2.2^{n}, ~\forall\, n\in \mathbb{N}$. Hence $~k~$ is not bounded, which is a contradiction. Therefore $~f~$ is not Lipschitz. 
 
 Thus we have the following proposition.
 
\begin{proposition}
$~(I,i)~$ is not the final co-algebra of $~F_{2}~$ on $~\mathbf{MS_{2}^{L}}~$. 
\end{proposition} 

\begin{example} 
\emph{
Consider the tri-pointed set $~\triangle=\{(x,0) ~/~ x\in [0~,~1]\}\cup\{(\frac{1}{2},  \frac{\sqrt{3}}{2}) \}~$, whose distinguished elements are given by $~T_{\triangle}=(\frac{1}{2},  \frac{\sqrt{3}}{2})~$ and $~L_{\triangle}=(0,  0)~,~R_{\triangle}=(1,  0)~,~$ and the metric is given by the euclidean metric on $~\mathbb{R}^{2}~$. }
\end{example}

Define $~e^{'}:\triangle\longrightarrow M_{3}\otimes \triangle~$ by 
$$e^{'}(x,y)=\left\{%
\begin{array}{lllll}
a\otimes (\frac{1}{2},  \frac{\sqrt{3}}{2}), &  \hbox{$(x,y)=(\frac{1}{2},  \frac{\sqrt{3}}{2})$;}\\\\
b\otimes (0,0), & \hbox{$x\in [0~\frac{1}{4}] ~\&~ y=0 $;} \\\\
b\otimes (4x-1,0), & \hbox{$x\in [\frac{1}{4}~\frac{1}{2}]~\&~ y=0 $;}\\\\
c\otimes (4x-2,0), & \hbox{$x\in [\frac{1}{2}~\frac{3}{4}]~\&~ y=0 $;}\\\\
c\otimes (1,0), & \hbox{$x\in [\frac{3}{4}~1]~\&~ y=0$.}
\end{array}%
\right.$$

This $~e^{'}~$ is a Lipschitz map with Lipschitz constant 2 and hence  $~(e^{'},\triangle)~$ is an $~F_{3}~$ co-algebra. 

Suppose $~(S,\psi)~$ is the final co-algebra. Then $~(\mathbb{S},\sigma)~$ is also a final co-algebra as they are isomorphic (see \cite{Bhat}).  Now, as in Example \ref{ExampleFinalCoalgebra}, there exists a unique Lipschitz map $~g:\triangle \longrightarrow ~\mathbb{S}$ such that the following diagram commutes.

$$\begin{array}{ccc}
\mathbb{S} & \xleftarrow{~~~~ \sigma^{-1} ~~~~} &  M_{3}\otimes \mathbb{S}\\\\
\uparrow{g} &  & \uparrow{M_{3}\otimes  g}\\\\
\triangle & \overrightarrow{~~~~ e^{'}~~~~} & M_{3}\otimes \triangle\\
\end{array}$$

By commutativity, $~g~$ must satisfy the following condition.\\\\
$~g(x,0)=0 , ~ x \in [0~\frac{1}{4}]~~,~~g(x,0)=1,~x\in [\frac{3}{4}~1]~$ and $$g(x,0)=\left\{%
\begin{array}{ll}
\displaystyle\frac{g(4x-1,0)}{2}, & \hbox{$x\in [\frac{1}{4}~\frac{1}{2}] $;} \\\\
\displaystyle\frac{1+g(4x-2,0)}{2}, & \hbox{$x\in [\frac{3}{4}~1]$.}
\end{array}%
\right.$$
Using these specific properties  of this mediating morphism,   $~g~$ will satisfy the properties given below. \begin{enumerate}[(a)]
	\item $~g(x,0)=0~, x\in I_{n}~,\forall~n\in\mathbb{N}$.
	\item $~g(x,0)=\displaystyle\frac{1}{2^{n}}~,x\in J_{n}~,\forall~n\in\mathbb{N}$.
\end{enumerate} 
 From these properties it follows, as in Example \ref{ExampleFinalCoalgebra}, that $\,g\,$ is not Lipschitz. Hence, neither $~(S,\psi)~$ nor $~(\mathbb{S},\sigma)~$ can be the final co-algebra. 

\begin{proposition}
$~(S,\psi)~$ is not the final co-algebra of $~F_{3}~$ on $~\mathbf{MS_{2}^{L}}~$.
 \end{proposition} 
 
 \section{Acknowledgement}
 The authors respectfully acknowledge the guidance and resourcefulness of  Professor Lawrence S. Moss, at Indiana University Bloomington.
 
\section*{References}
\bibliographystyle{model1-num-names}

\end{document}